\documentclass[12pt]{amsart}
\usepackage{amsmath,amssymb}
\usepackage[utf8,utf8x]{inputenc}
\usepackage[pagebackref=true,
pdftitle = {Free\ but\ not\ recursively\ free\ arrangements},
pdfkeywords = {arrangement\ of\ hyperplanes; free\ arrangement; Terao's\ conjecture;},
pdfsubject = {52C35; 20F55; 14N20; 13N15},
pdfauthor = {}]
{hyperref}
\usepackage{hyperref}
\usepackage{longtable}
\usepackage{graphicx}
\usepackage{pict2e}

\newtheorem{propo}{Proposition}[section]
\newtheorem{corol}[propo]{Corollary}
\newtheorem{theor}[propo]{Theorem}
\newtheorem{resul}[propo]{Result}
\newtheorem{lemma}[propo]{Lemma}
\newtheorem{conje}[propo]{Conjecture}

\theoremstyle{definition}
\newtheorem{defin}[propo]{Definition}

\theoremstyle{remark}

\newtheorem{oprob}[propo]{Open problem}

\numberwithin{equation}{section}

\newcommand{\CC }{\mathbb{C}}

\newcommand{\QQ }{\mathbb{Q}}
\newcommand{\ZZ }{\mathbb{Z}}

\newcommand{\Ac }{\mathcal{A}}

\newcommand{\Vc }{\mathcal{V}}

\DeclareMathOperator{\Aut}{Aut}

\newcommand{\IF }{\mathcal{IF}}
\newcommand{\RF }{\mathcal{RF}}
\DeclareMathOperator{\Der}{Der}
\DeclareMathOperator{\pdeg}{pdeg}
\DeclareMathOperator{\rk}{rk}

\newcommand{\lmultiset }{\{\hspace{-4.7pt}\{}
\newcommand{\rmultiset }{\}\hspace{-4.7pt}\}}

\title[Free, not recursively free and non rigid arrangements]
{Free, not recursively free\\ and non rigid arrangements}

\author{M.~Cuntz}
\address{Michael Cuntz,
Institut für Algebra, Zahlentheorie und Diskrete Mathematik,
Fakultät für Mathematik und Physik,
Leibniz Universität Hannover,
Welfengarten 1,
D-30167 Hannover, Germany}
\email{cuntz@math.uni-hannover.de}

\begin{document}

\begin{abstract}
We construct counterexamples to Yoshinaga's conjecture that every free arrangement is either inductively free or rigid in characteristic zero. The smallest example has $13$ hyperplanes, its intersection lattice has a one dimensional moduli space, and it is free but not recursively free.
\end{abstract}

\maketitle

\section{Introduction}

Let $\Ac$ be an arrangement of hyperplanes. To $\Ac$ we may associate its module of derivations $D(\Ac)$ and its intersection lattice $L(\Ac)$, see Section \ref{prel} for definitions. If $D(\Ac)$ is a free module, then $L(\Ac)$ satisfies certain strong combinatorial properties, for example that the characteristic polynomial of $L(\Ac)$ factorizes into degree one polynomials over the integers:
\begin{theor}[{\cite[Thm.\ 4.137]{OT}}] \label{Thm4137} 
Let $\Ac$ be a free arrangement and $e_1,\ldots,e_r$ be the degrees of a homogeneous basis of $D(\Ac)$. Then
\[ \chi(\Ac,t) = \prod_{i=1}^r (t-e_i). \]
\end{theor}
On the other hand, there are certain strong combinatorial properties of $L(\Ac)$ which imply that $D(\Ac)$ is a free module, for example if $L(\Ac)$ is \emph{inductively free}.
The holy grail in the field of arrangements is Terao's conjecture (see \cite{p-hT-80}, \cite{p-hT-81}, \cite{p-hT-83}, \cite{p-gZ-86}, \cite{p-gZ-89}, \cite{p-gZ-90}, \cite{OT}, \cite{p-Y-12}):
\begin{oprob}[Terao]\label{probter}
Is the freeness of the module of derivations of any arrangement over a fixed field $K$ a purely combinatorial property of its intersection lattice?
\end{oprob}
To be more precise: are there arrangements $\Ac_1,\Ac_2$ in the same vector space $V$ such that $L(\Ac_1)\cong L(\Ac_2)$, $D(\Ac_1)$ is free, and $D(\Ac_2)$ is not free?
We formulate Terao's Conjecture as an open problem here because there is a lot of evidence that it is wrong;
we quote Ziegler \cite{p-gZ-90}: ``We believe that Terao's conjecture over large fields is in fact false.'' In fact I suppose that Terao himself has doubts that the answer is `yes'.

Another formulation of Problem \ref{probter} is using the moduli space of all arrangements whose intersection lattice is a given lattice, see Section \ref{modulispace} for details:
\begin{oprob}
Does a lattice $L$ exist such that the moduli space $\Vc(L)$ contains a free and a non-free arrangement?
\end{oprob}

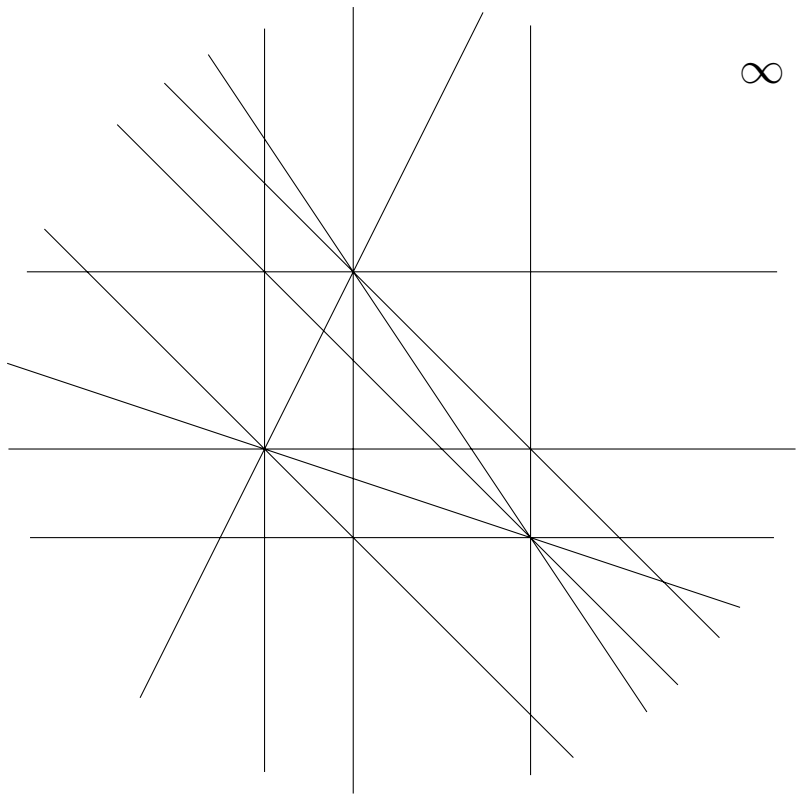
\begin{figure}
\begin{center}
\setlength{\unitlength}{0.75pt}
\begin{picture}(400,400)(100,200)
\moveto(275.403252247502313339499089644,598.481737195138434731455900375)
\lineto(275.403252247502313339499089644,201.518262804861565268544099625)
\moveto(498.481737195138434731455900375,375.403252247502313339499089644)
\lineto(101.518262804861565268544099625,375.403252247502313339499089644)
\moveto(439.167609485287577515505169793,256.360254559712843091676482869)
\lineto(156.360254559712843091676482869,539.167609485287577515505169793)
\moveto(470.504748702514523981268673470,295.462300245832978482848070577)
\lineto(100.873581185488167904693903570,418.672689418175097175039660544)
\moveto(230.681892697506519411315616269,587.603304874940835888685984162)
\lineto(230.681892697506519411315616269,212.396695125059164111314015838)
\moveto(364.845971347493901195866036393,589.195665912303604228199767981)
\lineto(364.845971347493901195866036393,210.804334087696395771800232018)
\moveto(487.603304874940835888685984162,330.681892697506519411315616270)
\lineto(112.396695125059164111314015838,330.681892697506519411315616270)
\moveto(386.440473505466700388082664215,219.644671439542132362732041699)
\lineto(119.644671439542132362732041699,486.440473505466700388082664215)
\moveto(489.195665912303604228199767982,464.845971347493901195866036393)
\lineto(110.804334087696395771800232018,464.845971347493901195866036393)
\moveto(460.106753515647283469601637540,280.142470079348931065763488497)
\lineto(180.142470079348931065763488497,560.106753515647283469601637540)
\moveto(340.870050249808744105836497484,595.779567352106762728540852073)
\lineto(167.898376268199836280669216832,249.836219388888947078206290770)
\moveto(423.506575597486519570920613016,242.690986322517591848733751335)
\lineto(202.294208758280284618416006238,574.509536581326944277490661501)
\thinlines
\strokepath
\put(470,560){\Large $\infty$}
\end{picture}
\end{center}
\caption{An arrangement with $13$ hyperplanes which is free but not recursively free and whose intersection lattice defines a one dimensional moduli space (the symbol $\infty$ stands for the line at infinity).\label{a13m1}}
\end{figure}

Not much is known about this moduli space regarding freeness; Yuzvinsky \cite{p-sY-93} proved that the free arrangements form an open subset in $\Vc(L)$ (see also \cite[Thm.\ 1.50]{p-Y-12} for the case of dimension three).
The first intuition one gets when working with these notions is, that either the moduli space is very big and the free arrangements are inductively free, or the moduli space is zero dimensional and all its elements are Galois conjugates (let us call such an arrangement \emph{rigid}).
Yoshinaga proposed to check the following conjecture which is stronger than Terao's conjecture since the property of being free is invariant under Galois automorphisms:
\begin{conje}[Yoshinaga, {\cite[p.\ 20, (11)]{p-Y-12}}]
\[ \{\text{Free arrangements}\} \subset \{\text{Inductively free}\} \cup \{\text{Rigid}\}.\]
\end{conje}

Since Terao's conjecture is trivial for arrangements in dimension two but still open in dimension three, we will concentrate on dimension three in this note.
A recent result by Abe illustrates that the case of dimension three is somewhat special:
\begin{theor}[{\cite[Corollary 1.5]{p-tA-13}}]
Let $\Ac,\Ac'$ be a deletion pair of arrangements in $K^3$, i.e.\ $\Ac'=\Ac\backslash\{H\}$ for some $H\in\Ac$. Then if the characteristic polynomials $\chi(\Ac)$ and $\chi(\Ac')$ have a common root, then both $\Ac$ and $\Ac'$ are free.
\end{theor}
Thus if there is a counterexample $L$ to Terao's Conjecture in dimension three, then $L$ cannot be inductively free, and a recursively free arrangement $\Ac$ with $L=L(\Ac)$ is very unlikely to exist.
Another recent progress is our counterexample with Hoge \cite{p-CH-13} to the conjecture (see \cite[3.6]{gZ-87}, \cite[4.3]{OT}, \cite[5]{p-S-99}) that free arrangements are recursively free in characteristic zero.
Although recursive freeness is not a combinatorial property of the intersection lattice, it thus appears to be a good idea to investigate the (probably small) class of arrangements which are free but not recursively free. The smallest example provided in \cite{p-CH-13} has $27$ hyperplanes. A natural question is thus how many hyperplanes such an example will have at least.

\begin{resul}\label{mainres}
In this note we present
\begin{enumerate}
\item a lattice $L$ such that the moduli space $\Vc_\CC(L)$ has dimension one and consists mostly of arrangements with 13 hyperplanes which are free but not recursively free,
\item and a lattice $L'$ such that the moduli space $\Vc_\CC(L')$ has dimension one and consists mostly of arrangements with 15 hyperplanes which are free but not recursively free.
\end{enumerate}
\end{resul}
We thus obtain:
\begin{corol}
Yoshinaga's conjecture about free, inductively free, and rigid arrangements is wrong.
\end{corol}
I would like to emphasize that to my knowledge the two lattices presented here are the only two known such examples. Lattices satisfying these strong properties appear to be extremely rare.

After posting the first version of this paper, I was informed by Professor Takuro Abe that Abe, Kawanoue, and Nozawa also found the arrangement with 13 hyperplanes independently, see \cite{p-AKN-14}.

\section{Preliminaries}\label{prel}

We shortly review the required notions, compare with \cite{OT}.

\begin{defin}
Let $\Ac$ be a central arrangement of hyperplanes, i.e.\ a finite set of linear hyperplanes in a fixed vector space $V$ over a field $K$.
Let $S=S(V^*)$ denote the symmetric algebra of the dual space $V^*$ of $V$.
We choose a basis $x_1,\ldots,x_r$ for $V^*$ and identify $S$ with
$K[x_1,\ldots,x_r]$ via the natural isomorphism $S\cong K[x_1,\ldots,x_r]$.
We write $\Der(S)$ for the set of derivations of $S$ over $K$.
It is a free $S$-module with basis $D_1,\ldots,D_r$ where $D_i$ is the usual derivation
$\partial/\partial x_i$.

A nonzero element $\theta\in\Der(S)$ is \emph{homogeneous of polynomial degree} $p$
if $\theta=\sum_{k=1}^r f_k D_k$ and $f_k\in S_p$ for $a\le k\le r$.
In this case we write $\pdeg \theta = p$.
Let $\Ac$ be an arrangement in $V$ with defining polynomial
\[ Q(\Ac) = \prod_{H\in\Ac} \alpha_H \]
where $H=\ker \alpha_H$, $\alpha_H\in V^*$. Define the \emph{module of $\Ac$-derivations} by
\[ D(\Ac) = \{\theta\in\Der(S)\mid \theta(Q(\Ac))\in Q(\Ac)S\}. \]
An arrangement $\Ac$ is called a \emph{free arrangement} if $D(\Ac)$ is a free
module over $S$.
If $\Ac$ is free and $\{\theta_1,\ldots,\theta_r\}$ is a homogeneous basis for
$D(\Ac)$, then $\pdeg\theta_1,\ldots,\pdeg\theta_r$ are called the \emph{exponents}
of $\Ac$ and we write
\[ \exp \Ac = \lmultiset \pdeg\theta_1,\ldots,\pdeg\theta_r\rmultiset, \]
where we use the notation $\lmultiset *\rmultiset$ to emphasize the fact that it is a multi\-set.
Remark that the exponents depend only on $\Ac$, thus by Theorem \ref{Thm4137} they depend only on $L(\Ac)$.
\end{defin}
\begin{defin}[{\cite[1.12-1.14]{OT}}]
Let $(\Ac,V)$ be an arrangement. We denote $L(\Ac)$ the set of all nonempty
intersections of elements of $\Ac$ including the empty intersection $V$.
For $X\in L(\Ac)$ define an arrangement $(\Ac^X,X)$ in $X$ by
\[ \Ac^X=\{X\cap H\mid X \not\subseteq H \mbox{ and } X\cap H\ne \emptyset\}.\]
We call $\Ac^X$ the {\it restriction} of $\Ac$ to $X$.
Let $H_0\in\Ac$. Let $\Ac'=\Ac\backslash\{H_0\}$ and let $\Ac''=\Ac^{H_0}$.
We call $(\Ac,\Ac',\Ac'')$ a {\it triple} of arrangements and $H_0$ the
{\it distinguished} hyperplane.
\end{defin}

\begin{theor}[Addition-Deletion, {\cite[Thm.~4.51]{OT}}]\label{adddel}
Suppose $\Ac\ne\emptyset$. Let $(\Ac,\Ac',\Ac'')$ be a triple.
Any two of the following statements imply the third:
\begin{eqnarray*}
\Ac \mbox{ is free with } \exp \Ac &=& \lmultiset b_1,\ldots,b_{r-1},b_r\rmultiset, \\
\Ac' \mbox{ is free with } \exp \Ac' &=& \lmultiset b_1,\ldots,b_{r-1},b_r-1\rmultiset, \\
\Ac'' \mbox{ is free with } \exp \Ac'' &=& \lmultiset b_1,\ldots,b_{r-1}\rmultiset.
\end{eqnarray*}
\end{theor}

Inspired by this theorem, one defines:

\begin{defin}[{\cite[Def.~4.53]{OT}}]\label{def:indfree}
The class $\IF$ of \emph{inductively free} arrangements is the smallest
class of arrangements which satisfies
\begin{enumerate}
\item the empty arrangement $\Phi_\ell$ of rank $\ell$ is in $\IF$ for $\ell\ge 0$,
\item\label{part2} if there exists $H\in \Ac$ such that $\Ac''\in\IF$, $\Ac'\in\IF$, and
$\exp \Ac''\subset\exp \Ac'$, then $\Ac\in\IF$.
\end{enumerate}
\end{defin}

Thus the property of an arrangement $\Ac$ of being inductively free is a 
combinatorial property of its intersection lattice $L(\Ac)$.
A class of arrangements which is bigger than the class of inductively free ones is:

\begin{defin}[{\cite[Def.~4.60]{OT}}]\label{def:recfree}
The class $\RF$ of \emph{recursively free} arrangements is the smallest
class of arrangements which satisfies
\begin{enumerate}
\item The empty arrangement $\Phi_\ell$ of rank $\ell$ is in $\RF$ for $\ell\ge 0$,
\item if there exists $H\in \Ac$ such that $\Ac''\in\RF$, $\Ac'\in\RF$, and
$\exp \Ac''\subset\exp \Ac'$, then $\Ac\in\RF$,
\item if there exists $H\in \Ac$ such that $\Ac''\in\RF$, $\Ac\in\RF$, and
$\exp \Ac''\subset\exp \Ac$, then $\Ac'\in\RF$.
\end{enumerate}
\end{defin}

We need the following simple lemma (see also \cite[Lemma 3.1]{p-CH-13} or \cite[Cor.\ 2.18]{p-HR12}):

\begin{lemma}\label{strnotindfree}
Let $\Ac$ be a free arrangement in $K^3$ with exponents $\lmultiset 1$, $e$, $f\rmultiset$.
If $|\Ac^H| \notin\{e+1,f+1\}$ for all $H\in\Ac$, then $\Ac$ is not inductively free.
\end{lemma}

For $\alpha=(a_1,\ldots,a_r)\in K^r$ we will write
\[ \ker \alpha := \left\{ (x_1,\ldots,x_r)\in K^r \mid \sum_{i=1}^r a_i x_i =0\right\}. \]

\section{Intersection lattices and moduli spaces}\label{modulispace}

Given a matrix $A\in K^{r\times n}$, let $\Ac(A)$ be the arrangement
\[ \Ac(A) := \{ \ker \alpha \mid \alpha \text{ is a column vector of } A, \:\: \alpha\ne 0 \}. \]
Thus the map $A \mapsto \Ac(A)$ is a surjection onto the set of arrangements in $K^r$ with at most $n$ hyperplanes.

Assume that $(L,\le)$ is a graded lattice, i.e.\ all maximal chains in $L$ have the same length. This implies the existence of a map $\rk : L \rightarrow \ZZ$ such that for fixed $n\in\ZZ$ the set $\{U\in L\mid \rk(U)=n\}$ is an antichain.

For a fixed field, the set $\Ac(L)$ consisting of those arrangements $\Ac$ with $L(\Ac)\cong L$ is not an algebraic variety.
However, in dimension three it is a constructible set:
\begin{propo}
Let $L$ be a graded lattice with $\rk(L)=\{0,\ldots,3\}$, $K$ be a field, and
\[ \Vc(L) := \{ A \in K^{3\times n} \mid |\Ac(A)|=n, \:\: L(\Ac(A))\cong L\} \]
where $n = \{U\in L\mid \rk(U)=1\}$.

Then $\Vc(L)$ is a constructible set, i.e., there exist varieties $V_1,V_2$ in $K^{3\times n}$ such that $\Vc(L) = V_1\backslash V_2$.
Moreover, $\{\Ac(A)\mid A\in \Vc(L)\}$ is the set of all arrangements $\Ac\in K^r$ with $L(\Ac)\cong L$.
\end{propo}
\begin{proof}
The isomorphism $L(\Ac(A))\cong L$ implies that we know exactly which subsets of the columns of $A$ are linearly independent or linearly dependent. These dependencies and independencies translate into equations of the form $\det(v_1,v_2,v_3)=0$ or $\det(v_1,v_2,v_3)\ne 0$ respectively, thus contribute to a variety $V_1$ or $V_2$ respectively.
\end{proof}

If $L$ is a poset, we write $\Aut(L)$ for the set of automorphisms of posets, i.e.\ the set of bijections preserving the relation in the poset.

\section{Examples}

\subsection{Arrangements with $13$ hyperplanes}

Let $R\subseteq \ZZ[t]^3$ be the set
\begin{eqnarray*}
R &=& \{ (1,0,0),(0,1,0),(0,0,1),(1,0,-1),(0,1,-1),(1,1,-1), \\
&& (1,0,-t),(0,1,-t),(1,1,-t),(1,1,-t-1), \\
&& (t,1,-t),(1,-t+1,-1),(t-1,t,-t^2) \}.
\end{eqnarray*}
We first consider the arrangement $\Ac_t = \{\ker_{\QQ(t)} \alpha\mid \alpha \in R\}$ as an arrangement in the space $\QQ(t)^3$.
Labeling the hyperplanes of $\Ac_t$ by $1,\ldots,13$ (in the ordering given by the definition of $R$) and the one dimensional elements by $1,\ldots,30$, the intersection lattice $L(\Ac_t)$ can be described by the following sequence of lists:
\begin{align*}
&\left[1, 2, 3, 4, 5, 6\right],\left[1, 7, 8, 9, 10, 11\right],\left[2, 7, 12, 13, 14, 15\right],\\
&\left[2, 8, 16, 17, 18, 19\right],\left[3, 7, 16, 20, 21, 22\right],\left[3, 8, 12, 23, 24, 25\right],\\
&\left[2, 9, 20, 23, 26, 27\right],\left[4, 7, 17, 24, 26, 28\right],\left[4, 9, 12, 18, 21, 29\right],\\
&\left[5, 10, 12, 17, 20, 30\right],\left[4, 8, 13, 22, 27, 30\right],\left[6, 8, 14, 20, 28, 29\right],\\
&\left[4, 11, 15, 19, 20, 25\right],
\end{align*}
for example the third list
encodes that (the projective points) $2$, $7$, $12$, $13$, $14$, $15$ lie on hyperplane number three.
From the intersection lattice we can compute
\[ \chi_{\Ac_t}(x) = (x-1)(x-6)^2, \]
and we see that $\Ac_t$ is not inductively free by Lemma \ref{strnotindfree}.

For any field extension $K/\QQ$ and each $\omega \in K$, we also have a corresponding arrangement
\[ \Ac_\omega:=\{ \ker_K \alpha(\omega) \mid \alpha\in R\} \]
in $K^3$.
It is easy to check that $L(\Ac_\omega)\ne L(\Ac_t)$ if and only if
\[ \omega \in Z := \left\{ -1,0,\frac{1}{2},1,2,\frac{1\pm \sqrt{-3}}{2} \right\}. \]
If $\omega\in \{0,1,\frac{1\pm \sqrt{-3}}{2}\}$ then $|\Ac_t|\ne|\Ac_\omega|$.
For $\omega\in \{-1,\frac{1}{2},2\}$ we have $|\Ac_t|=|\Ac_\omega|$, but the intersection lattices are different.
See Figure \ref{a13m1} for a picture in the real projective plane of one of the $\Ac_\omega$ with $\omega\notin Z$.

Moreover, using the algorithm of \cite{p-C10b} we get
\[ \Vc_\CC(L(\Ac_t)) = \{ \Ac_\omega \mid \omega \in \CC\backslash Z\}. \]
Thus the moduli space of $L(\Ac_t)$ is one dimensional. A short computation with Singular shows that every arrangement in $\Vc_\CC(L(\Ac_t))$ is free. Further, $\Ac_t$ is free but not recursively free.
It is remarkable that the arrangements $\Ac_\omega$, $\omega\in \{-1,\frac{1}{2},2\}$ are free as well, but their exponents are $\lmultiset 1,5,7 \rmultiset$.

The symmetry group of $L(\Ac_t)$ is
\[ \Aut(L(\Ac_t)) \cong \ZZ/3\ZZ \rtimes S_3. \]

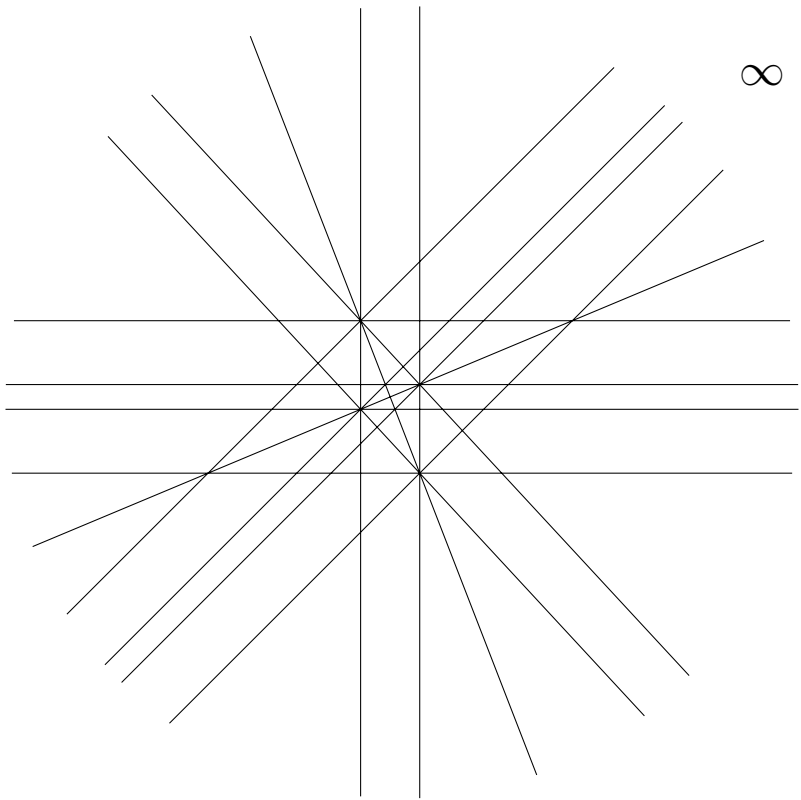
\begin{figure}
\begin{center}
\setlength{\unitlength}{0.75pt}
\begin{picture}(400,400)(100,200)
\moveto(496.773982019981493284007282848,364.222912360003364857453221300)
\lineto(103.226017980018506715992717152,364.222912360003364857453221300)
\moveto(499.799899899874824737086828568,408.944271909999158785636694675)
\lineto(100.200100100125175262913171431,408.944271909999158785636694675)
\moveto(495.715111622386654550161466794,441.177604139103725326976197484)
\lineto(104.284888377613345449838533206,441.177604139103725326976197484)
\moveto(462.003080212687308662930982098,517.281720662691514734747508724)
\lineto(182.718279337308485265252491277,237.996919787312691337069017902)
\moveto(441.421356237309504880168872421,541.421356237309504880168872421)
\lineto(158.578643762690495119831127579,258.578643762690495119831127579)
\moveto(406.952750824954505690978417484,569.000322754056281424622386528)
\lineto(130.999677245943718575377613472,293.047249175045494309021582516)
\moveto(308.944271909999158785636694675,599.799899899874824737086828568)
\lineto(308.944271909999158785636694675,200.200100100125175262913171431)
\moveto(367.958188185589464486680994098,211.899801545739966944654699262)
\lineto(223.491767387609719016567995648,584.787689910687449561575785165)
\moveto(444.811903076327848709446111418,262.052500104524442500854566083)
\lineto(173.745865467501343672573169961,555.112518880487948495621622922)
\moveto(499.968602029388036590795160647,396.456244589107931398792724109)
\lineto(100.031397970611963409204839353,396.456244589107931398792724109)
\moveto(432.492660930138327285243171407,549.818873309244309090703667076)
\lineto(150.181126690755690909296332923,267.507339069861672714756828593)
\moveto(422.256045515541062845432718669,241.717154009343286710386751204)
\lineto(151.716113756179965018531360123,534.208379322715313821763695360)
\moveto(482.566108371663406313081719590,481.667717453263334783894740716)
\lineto(113.729449702885217465211784879,327.176363095420508928667332648)
\moveto(279.130032210001949593332228440,598.908130664496578018818014653)
\lineto(279.130032210001949593332228440,201.091869335503421981181985346)
\thinlines
\strokepath
\put(470,560){\Large $\infty$}
\end{picture}
\end{center}
\caption{An arrangement $A_\omega$ with $15$ hyperplanes.}
\end{figure}

\subsection{Arrangements with $15$ hyperplanes}

Now let $R\subseteq \ZZ[t]^3$ be the set
\begin{eqnarray*}
R &=& \{ (1,0,0),(1,1,0),(1,0,1),(1,1,1),(1,t,1),(0,1,0),(2,1,1),\\
&& (t+1,t,1),(t+1,1,1),(2t,t,1),(1,-t+1,1),\\
&& (-3t+1,t^2-3t+1,-t),(3t-1,t,t),\\ && (-3t+1,-t^2,-t),(3t-1,2t-1,t) \}.
\end{eqnarray*}
With the analogous notations $\Ac_t$, $\Ac_\omega$ and similar computations as in the example with $13$ hyperplanes, we get:
\begin{enumerate}
\item If $\omega\in \{0, 1, -1, \frac{1}{2}\}$ then $|\Ac_t|\ne|\Ac_\omega|$.
\item If $\omega\in \{\frac{3}{2}\pm \sqrt{2},(-1\pm \sqrt{5})/2\}$ then $|\Ac_t|=|\Ac_\omega|$, but the intersection lattices are different.
\item $\Vc_\CC(L(\Ac_t)) = \{ \Ac_\omega \mid \omega \in \CC\backslash Z\}$ for $$Z=\left\{0, 1, -1, \frac{1}{2},\frac{3}{2}\pm \sqrt{2},\frac{-1\pm \sqrt{5}}{2}\right\}.$$
\item The arrangement in $\Ac_t$ is free, but not recursively free with exponents $\lmultiset 1,7,7\rmultiset$.
\item If $\omega\in \{\frac{3}{2}\pm \sqrt{2},(-1\pm \sqrt{5})/2\}$ then $\Ac_\omega$ is free with exponents $\lmultiset 1,5,9\rmultiset$.
\item The moduli space of $L(\Ac_t)$ is one dimensional.
\item The symmetry group of $L(\Ac_t)$ is $\Aut(L(\Ac_t)) \cong \ZZ/2\ZZ \wr S_3$, thus isomorphic to the reflection group of type $B_3$.

\end{enumerate}

\def\cprime{$'$}
\providecommand{\bysame}{\leavevmode\hbox to3em{\hrulefill}\thinspace}
\providecommand{\MR}{\relax\ifhmode\unskip\space\fi MR }
\providecommand{\MRhref}[2]{%
  \href{http://www.ams.org/mathscinet-getitem?mr=#1}{#2}
}
\providecommand{\href}[2]{#2}

\end{document}